\documentclass[12pt,reqno]{amsart}
\usepackage{amsmath,amssymb,amsthm,fullpage, graphicx, chngcntr,float,mathtools,hyperref}
\newtheorem{theorem}{Theorem}[section]
\newtheorem{lemma}[theorem]{Lemma}

\theoremstyle{definition}
\newtheorem{definition}[theorem]{Definition}

\newcommand{\N}{\mathbb{N}}

\newcommand{\C}{\mathbb{C}}
\newcommand{\2}{\C^2}
\newcommand{\3}{\mathbb{S}^3}

\newcommand{\conj}[1]{\overline{#1}}

\renewcommand{\aa}{\alpha_1}
\newcommand{\ab}{\alpha_2}

\newcommand{\ba}{\beta_1}
\newcommand{\bb}{\beta_2}
\renewcommand{\l}{\mathcal{L}}
\newcommand{\lbar}{\smash[b]{\conj{\mathcal{L}}}}
\newcommand{\boxb}{\square_b}
\newcommand{\boxbt}{\square_b^t}
\newcommand{\restrict}{\raise-.3ex\hbox{\ensuremath|}}

\newcommand{\hpqs}[2]{\mathcal{H}_{#1,#2}(\3)}
\newcommand{\hks}[1]{\mathcal{H}_{#1}(\3)}

\newcommand{\hpq}[3]{\mathcal{H}_{#1,#2}(#3)}
\newcommand{\hk}[2]{\mathcal{H}_{#1}(#2)}
\newcommand{\ppq}[3]{\mathcal{P}_{#1,#2}(#3)}
\newcommand{\pk}[2]{\mathcal{P}_{#1}(#2)}

\newcommand{\pderiv}[1]{\dfrac{\partial}{\partial #1}}

\DeclareMathOperator{\spec}{spec}
\DeclareMathOperator{\essspec}{essspec}
\DeclareMathOperator{\Span}{span}

\counterwithin{figure}{section}

\title{Spectrum of  the Kohn Laplacian on the Rossi sphere}

\author{Tawfik Abbas}
\address[Tawfik Abbas]{Michigan State University, Department of Mathematics, East Lansing, MI 48824, USA}
\email{abbastaw@msu.edu}

\author{Madelyne M. Brown}
\address[Madelyne M. Brown]{Bucknell University, Department of Mathematics, Lewisburg, PA 17837, USA}
\email{mmb021@bucknell.edu}

\author{Ravikumar Ramasami}
\address[Ravikumar Ramasami]{University of Michigan--Dearborn, Department of 
	Mathematics \& Statistics, Dearborn, MI 48128, USA}
\email{rramasam@umich.edu}

\author{Yunus E. Zeytuncu}
\address[Yunus E. Zeytuncu]{University of Michigan--Dearborn, Department of 
	Mathematics \& Statistics, Dearborn, MI 48128, USA}
\email{zeytuncu@umich.edu}

\thanks{This work is supported by NSF (DMS-1659203) and the University of Michigan--Dearborn. The work of the fourth author is also partially supported by a grant from the Simons Foundation (\#353525).}

\begin{document}
\maketitle

    \begin{abstract}
    We study the spectrum of the Kohn Laplacian $\boxbt$ on the Rossi example $(\3, \l_t)$. In particular we show that $0$ is in the essential spectrum of $\boxbt$, which yields another proof of the global non-embeddability of the Rossi example.
    \end{abstract}

%\tableofcontents

\section{Introduction}
 
   When is an abstract CR-manifold globally CR-embeddable into $\C^N$? Rossi showed that the CR-manifold $(\3, \l_t)$ is not CR-embeddable \cite{Rossi65}, where $\3$ is the 3-sphere in $\2$,
    \[
        \l_t = \conj{z_1} \pderiv{z_2} - \conj{z_2} \pderiv{z_1} + \conj{t} \left( z_1 \pderiv{\conj{z_2}} - z_2 \pderiv{\conj{z_1}} \right),
    \] 
    and $|t| < 1$. In the case of strictly pseudoconvex CR-manifolds Boutet de Monvel proved that if the real dimension of the manifold is at least $5$, then it can always be globally CR-embedded into $\C^N$ for some $N$ \cite{Monvel75}. Later Burns approached this problem in the $\overline{\partial}$ context and showed that if the tangential operator $\conj{\partial}_{b,t}$ has closed range and the Szeg\"{o} projection is bounded, then the CR-manifold is CR-embeddable into $\C^N$ \cite{Burns79}. Later in 1986, Kohn showed that CR-embeddability is equivalent to showing that the tangential Cauchy-Riemann operator $\conj{\partial}_{b,t}$ has closed range \cite{Kohn85}. We refer to \cite[Chapter 12]{ChenShaw01PDE} for a full account of these results and also to \cite{Boggess91CR} for general theory of CR-manifolds. 
    
In the setting of the Rossi example, as an application of the closed graph theorem, $\conj{\partial}_{b,t}$ has closed range if and only if the Kohn Laplacian 
$$\boxbt=-\l_t \frac{1 + |t|^2}{(1 - |t|^2)^2} \lbar_t$$ 
has closed range, see \cite[0.5]{BurnsEpstein}. Furthermore, the closed range property is equivalent to the positivity of the essential spectrum of $\boxbt$, see \cite{Fu2005} for similar discussion. In this note we tackle the problem of embeddability, from the perspective of spectral analysis. In particular, we show that 0 is in the essential spectrum of $\boxbt$, so the Rossi sphere is not globally CR-embeddable in $\C^N$. This provides a different approach to the results in \cite{Burns79,Kohn85}.
    
    We start our analysis with the spectrum of $\boxbt$. We utilize spherical harmonics to construct finite dimensional subspaces of $L^2(\3)$ such that $\boxbt$ has tridiagonal matrix representations on these subspaces. We then use these matrices to compute eigenvalues of $\boxbt$. We also present numerical results obtained by \textit{Mathematica} that motivate most of our theoretical results. We then present an upper bound for small eigenvalues and we exploit this bound to find a sequence of eigenvalues that converge to 0. 
    
    In addition to particular results in this note, our approach can be adopted to study possible other perturbations of the standard CR-structure on the 3-sphere, such as in \cite{BurnsEpstein}. Furthermore, our approach also leads some information on the growth rate of the eigenvalues and possible connections to finite-type (order of contact with complex varieties) results similar to the ones in \cite{Fu2008}. We plan to address these issues in future papers.

%%%%%%%%%%%%%%%%%%%%%%%%%%%%%%%%%%%%%%%%%%%%%%%%%%

    \section{Analysis of $\boxb$ on $\hpqs{p}{q}$} %Tawfik
    
%%%%%%%%%%%%%%%%%%%%%%%%%%%%%%%%%%%%%%%%%%%%%%%%%%
    
    \subsection{Spherical Harmonics}
    We start with a quick overview of spherical harmonics, we refer to \cite{Axler13Harmonic} for a detailed discussion.
    We will state the relevant theorems on $\2$ and $\3 \subseteq \2$. A polynomial in $\2$ looks like
    \[
        p(z, \conj{z}) = \sum_{\alpha, \beta} c_{\alpha, \beta} z^\alpha \conj{z}^\beta
    \]
    where each $c_{\alpha, \beta} \in \C$, and $\alpha$, $\beta$ are multi-indices. That is, $\alpha = (\aa, \ab), z^\alpha = z_1^{\aa} z_2^{\ab}$, and $|\alpha| = \aa + \ab$.

    We denote the space of all homogeneous polynomials on $\2$ of degree $m$ by $\pk{m}{\2}$, and we let $\hk{m}{\2}$ denote the subspace of $\pk{m}{\2}$ that consists of all harmonic homogeneous polynomials on $\2$ of degree $m.$ We use $\pk{m}{\3}$ and $\hk{m}{\3}$ to denote the restriction of $\pk{m}{\2}$ and $\hk{m}{\2}$ onto $\3$.  We denote the space of complex homogenous polynomials on $\2$ of bidegree $p,q$ by $\ppq{p}{q}{\2}$, and those polynomials that are homogeneous and harmonic by $\hpq{p}{q}{\2}$. As before, we denote $\ppq{p}{q}{\3}$ and $\hpq{p}{q}{\3}$ as the polynomials of the previous spaces, but restricted to $\3$.
We recall that on $\2$, the Laplacian is defined as $\Delta = 4(\frac{\partial^2}{\partial z_1 \partial \conj{z_1}} + \frac{\partial^2}{\partial z_2 \partial \conj{z_2}})$. As an example, the polynomial $z_1 \conj{z_2} - 2 z_2 \conj{z_1} \in \ppq{1}{1}{\2}$, and $z_1 \conj{z_2}^2 \in \hpq{1}{2}{\C^2}$. We take our first step by stating the following theorem.
	
	\begin{theorem}\cite[Theorem 5.1]{Axler13Harmonic} \label{poissonpoly}
	If $p$ is a polynomial on $\2$ of degree $m,$ then 
	\[
	    P[p\restrict_{\3}] = (1 - |z|^2)q + p
	\]
	for some polynomial $q$ of degree at most $m-2.$
	\end{theorem}

	This theorem highlights how the Poisson integral of an $m$ degree polynomial on $\3$ can be represented by a polynomial decomposition. As the Poisson integral yields a harmonic polynomial, the polynomial decomposition will be harmonic. 
	
	    Similarly, we have the following decomposition for the space of homogeneous polynomials into a space of harmonic polynomials and a space of homogeneous polynomials with a factor of $|z|^2$.
    \begin{theorem}\cite[Theorem 5.5]{Axler13Harmonic} \label{decomposepksingle}
    If $m \geq 2,$ then 
    \[
        \pk{m}{\2} = \hk{m}{\2} \oplus |z|^2\pk{m-2}{\2},
    \]
    and
    \[\ppq{p}{q}{\2} = \hpq{p}{q}{\2} \oplus |z|^2 \ppq{p-1}{q-1}{\2}.\]
    \end{theorem}
    
    By applying the previous statement multiple times to the homogeneous part of a polynomial decomposition, we arrive at the following theorem. 
    
    \begin{theorem}\cite[Theorem 5.7]{Axler13Harmonic} \label{decomposepkmultiple}
    	Every $p \in \pk{m}{\2}$ can be uniquely written in the form 
    	\[
    	    p = p_{m} + |z|^2 p_{m-2} + ... + |z|^{2k}p_{m-2k}
    	\]
    	where $k = [\frac{m}{2}]$ and each $p_{i} \in \hk{m}{\2},$ where $[x]$ means the nearest integer to $x$. 
	\end{theorem}

    This yields to the following decomposition of the space of square integrable functions on $\3$.
    
    \begin{theorem}\cite[Theorem 5.12]{Axler13Harmonic} \label{L2S3Directsum}
    $L^2 (\3) = \bigoplus^\infty_{m = 0} \hk{m}{\3}$.
    \end{theorem}
    
    The previous theorem is essential to the spectral analysis of $\boxbt$ on $L^2(\3)$ since it decomposes the infinite dimensional space $L^2(\3)$ into finite dimensional pieces, which is necessary for obtaining the matrix representation of $\boxbt$. In order to get such a matrix representation, we need a method for obtaining a basis for $\hks{k}$. Theorem \ref{basishk} presents a method to do so for  $\hk{m}{\2}$ and Theorem \ref{basishpq} presents a method for $\hpq{p}{q}{\2}$. The dimension of the matrix representation on a particular $\hks{m}$ is the dimension of the subspace $\hks{m}$, which is given below and analogously given for $\hpq{p}{q}{\2}$.
    
    \begin{theorem}\cite[Proposition 5.8]{Axler13Harmonic} \label{sizehk}
    If $m \geq 2,$ then 
    \[
        \dim \hk{m}{\2} = \binom{n + m -1}{n-1} - \binom{n + m - 3}{n - 1}, 
    \]
    \[\dim \ppq{p}{q}{\2} = (p+1)(q+1),\]
    and
     \begin{gather*}
            \dim \hpq{p}{q}{\2} = p + q + 1 \\
            \dim \hk{k}{\2} = (k + 1)^2.
        \end{gather*}
    \end{theorem}
    
Now we present a method to obtain explicit bases of spaces of spherical harmonics. These bases play an essential role in explicit calculations in the next section. Here, $K$ denotes the Kelvin trasform,
\[Kg(z)=|z|^{-2}g\left(\frac{z}{|z|^2}\right)
.\]    
   
    \begin{theorem}\cite[Theorem 5.25]{Axler13Harmonic} \label{basishk}
    If $ n > 2$ then the set \[\{K[D^\alpha |z|^{-2}]:|\alpha| = m \text{ and } \alpha_1 \leq 1 \} \] is a vector space basis of $\hk{m}{\2},$ and the set \[\{D^\alpha |z|^{-2} : |\alpha| = m \text{ and } \aa \leq 1\}\] is a vector space basis of $\hk{m}{\3}.$
    \end{theorem}

    It follows from the previous definition that the homogenous polynomials of degree $k$ can be written as the sum of polynomials of bidegree $p,q$ such that $p+q=k$.
     \begin{theorem} \label{pkissumppq}
        $\pk{k}{\2} = \bigoplus_{p+q=k} \ppq{p}{q}{\2}$.
    \end{theorem}
    
       Analogous to the version in Theorem \ref{basishk}, we use the following method to construct an orthogonal basis for $\hpq{p}{q}{\2}$ and $\hpqs{p}{q}$. 
    \begin{theorem} \label{basishpq}
        The set
        \[
            \left\{ K[\conj{D}^\alpha D^\beta |z|^{-2}] \;\bigg\lvert\; |\alpha| = p, |\beta| = q, \alpha_1 = 0 \text{ or } \beta_1 = 0 \right\}
        \]
        is a basis for $\hpq{p}{q}{\2}$, and the set
        \[
            \left\{ \conj{D}^\alpha D^\beta |z|^{-2} \;\bigg\lvert\; |\alpha| = p, |\beta| = q, \alpha_1 = 0 \text{ or } \beta_1 = 0 \right\}
        \]
        is an orthogonal basis for $\hpqs{p}{q}$.
    \end{theorem}

%%%%%%%%%%%%%%%%%%%%%%%%%%%%%%%%%%%%%%%%%%%%%%%%%%

    	\subsection{$\boxb$ on $\hpqs{p}{q}$}
	Before we study the operator $\boxbt$, we first need some background on a simpler operator we call $\boxb$. It arises from the CR-manifold $(\3, \l)$, and is defined as
	\[
		\boxb = -\l \lbar.
	\]
We note that this CR-structure is induced from $\2$ and this manifold is naturally embedded. By the machinery above we can compute the eigenvalues of $\boxb$.

	\begin{theorem}
		\label{boxbeigenvalue}
		Suppose $f \in \hpqs{p}{q}$. Then
		\[
			\boxb f = (pq + q)f.
		\]
	\end{theorem}
	\begin{proof}
		Expanding the definition, we get
		\begin{alignat*}{2}
			\boxb &= &&-\left(\conj{z_2} \pderiv{z_1} - \conj{z_1} \pderiv{z_2}\right)\left(z_2 \pderiv{\conj{z_1}} - z_1 \pderiv{\conj{z_2}}\right) \\
			&= &&-\conj{z_2} \pderiv{z_1} \left(z_2 \pderiv{\conj{z_1}} - z_1 \pderiv{\conj{z_2}}\right) + \conj{z_1} \pderiv{z_2} \left(z_2 \pderiv{\conj{z_1}} - z_1 \pderiv{\conj{z_2}}\right) \\
			&= &&-z_2 \conj{z_2} \frac{\partial^2}{\partial z_1 \partial \conj{z_1}} + \conj{z_2} \pderiv{\conj{z_2}} + z_1 \conj{z_2} \frac{\partial^2}{\partial z_1 \partial \conj{z_2}} \\
			&&&-z_1 \conj{z_1} \frac{\partial^2}{\partial z_2 \partial \conj{z_2}} + \conj{z_1} \pderiv{\conj{z_1}} + z_2 \conj{z_1} \frac{\partial^2}{\partial z_2 \partial \conj{z_1}}
		\end{alignat*}
		Now, let $f \in \hpqs{p}{q}$. Since $f$ is harmonic, we know that $\frac{\partial^2}{\partial z_1 \partial \conj{z_1}} = -\frac{\partial^2}{\partial z_2 \partial \conj{z_2}}$. Substituting, we get
		\begin{alignat*}{1}
			&= z_2 \conj{z_2} \frac{\partial^2}{\partial z_2 \partial \conj{z_2}} + \conj{z_2} \pderiv{\conj{z_2}} + z_1 \conj{z_2} \frac{\partial^2}{\partial z_1 \partial \conj{z_2}} \\
			&+z_1 \conj{z_1} \frac{\partial^2}{\partial z_1 \partial \conj{z_1}} + \conj{z_1} \pderiv{\conj{z_1}} + z_2 \conj{z_1} \frac{\partial^2}{\partial z_2 \partial \conj{z_1}}
		\end{alignat*}
		Since $f$ is a polynomial and $\boxb$ is linear, it suffices to show that if $f = z^\alpha \conj{z}^\beta = z_1^{\aa} z_2^{\ab} \conj{z_1}^{\ba} \conj{z_2}^{\bb}$, where $\aa + \ab = p$ and $\ba + \bb = q$, then the claim holds. Using the expansion we got, each derivative expression simply becomes a multiple of $f$. So we get
		\begin{align*}
			\boxb f &= (\ab \bb + \bb + \aa \bb + \aa \ba + \ba + \ab \ba)f \\
			&= ((\aa + \ab)(\ba + \bb) + (\ba + \bb))f \\
			&= (pq + q)f
		\end{align*}
		and we are done.
	\end{proof}
	In a similar manner, we can show that $-\lbar \l f = (pq + p)f$. For this case, we actually have that $\spec(\boxb) = \{pq + q \mid p,q \in \N\}$, so $0 \notin \essspec(\boxbt)$ since it is not an accumulation point of the set above. %(It occurs infinitely often, but we ignore it. Why? It is in the kernel.)

%%%%%%%%%%%%%%%%%%%%%%%%%%%%%%%%%%%%%%%%%%%%%%%%%%		
    
    \section{Experimental Results in Mathematica}%Maddie
    
     Using the symbolic computation environment provided by \textit{Mathematica}, we were able to write a program to streamline our calculations\footnote{Our code for this and the other symbolic computations described below is available on our website at \url{https://sites.google.com/a/umich.edu/zeytuncu/home/publ}}. We implemented the algorithm provided in Theorem \ref{basishpq} to construct the vector space basis of $\hks{k}$ for a specified $k$. As an example, our code produced the following basis of $\hks{3}$:
    \begin{multline*}
        \{-6 \overline{z_2}^3, 
        -6 \overline{z_1} \overline{z_2}^2,
        -6 \overline{z_1}^2 \overline{z_2},
        -6\overline{z_1}^3, 
        4 z_1\overline{z_1} \overline{z_2}-2 z_2\overline{z_2}^2,
        2 z_1 \overline{z_1}^2-4 z_2 \overline{z_1} \overline{z_2},
        -6 z_2 \overline{z_1}^2,
        -6 z_1\overline{z_2}^2, \\
        4 z_1 z_2 \overline{z_1}-2 z_2^2 \overline{z_2},
        -6 z_2^2 \overline{z_1},
        2 z_1^2 \overline{z_1}-4 z_1 z_2 \overline{z_2},
        -6 z_1^2 \overline{z_2}, 
        -6 z_2^3,
        -6 z_1 z_2^2,
        -6 z_1^2 z_2,
        -6 z_1^3 \}.
    \end{multline*}
%     Our code for this and the other symbolic computations described below is available on our website at \url{} under BoxBT$\_$Experiment.nb.%This website isn't actually public, so where should we put this file?
Now, with the basis for $\hks{k}$, the matrix representation of $\boxbt$ on $\hks{k}$ can be computed for each $k$. In particular, we used this program to construct the matrix representations for $1\leq k\leq 12$. For a specific $k$, the code applies $\boxbt$ to each basis element of $\hks{k}$ obtained by the results in the previous sections. Then, using the inner product defined by,
    \[
        \langle f, g \rangle= \int_{\3} f \overline{g} \,d\sigma,
    \]
    where $\sigma$ is the standard surface area measure, the software computes $ \langle \boxbt f_i,f_j\rangle$, where $f_i,f_j$ are basis vectors for $\hks{k}$. With these results, \textit{Mathematica} yields the matrix representation for the imputed value of $k$. For example, the program produced the matrix representation for $k=3$ seen in Figure \ref{fig:matrix}. Since each entry had a common normalization factor,
    \[
        h=\frac{1+|t|^2}{(1-|t|^2)^2},
    \]
    this constant has been factored out.
    \begin{figure}[h] 
        \resizebox{\textwidth}{!}{$\displaystyle
       h \begin{pmatrix}
            3 & 0 & 0 & 0 & 0 & 0 & 0 & 0 & 0 & 0 & 0 & -6\overline{t} & 0 & 0 & 0 & 0\\
            0 & 3 & 0 & 0 & 0 & 0 & 0 & 0 & 0 & 0 & 6\overline{t} & 0 & 0 & 0 & 0 & 0\\
            0 & 0 & 3 & 0 & 0 & 0 & 0 & 0 & -6\overline{t} & 0 & 0 &0 & 0 & 0 & 0 & 0\\
            0 & 0 & 0 & 3 & 0 & 0 & 0 & 0 &0 & -6\overline{t} & 0 &0 & 0 & 0 & 0 & 0 \\
            0 & 0 & 0 & 0 & 4+3|t|^2 & 0 & 0 & 0 & 0 & 0 & 0 & 0 & 0 & 0 & -2\overline{t} & 0 \\
            0 & 0 & 0 & 0 & 0 & 4+3|t|^2 & 0 & 0 & 0 & 0 & 0 & 0 & 0 & 2\overline{t} & 0 & 0 \\
            0 & 0 & 0 & 0 & 0 & 0 & 4+3|t|^2 & 0 & 0 & 0 & 0 & 0 & -2\overline{t} & 0 & 0 & 0 \\
            0 & 0 & 0 & 0 & 0 & 0 & 0 & 4+3|t|^2 & 0 & 0 & 0 & 0 & 0 & 0 & 0 & -2\overline{t} \\
            0 & 0 & -2t & 0 & 0 & 0 & 0 & 0 & 3+4|t|^2 & 0 & 0 & 0 & 0 & 0 & 0 & 0\\
            0 & 0 & 0 & -2t & 0 & 0 & 0 & 0 & 0 & 3+4|t|^2 & 0 & 0 & 0 & 0 & 0 & 0\\
            0 & 2t & 0 & 0 & 0 & 0 & 0 & 0 & 0 & 0 & 3+4|t|^2 & 0 & 0 & 0 & 0 & 0\\
            -2t & 0 & 0 & 0 & 0 & 0 & 0 & 0 & 0 & 0 & 0 & 3+4|t|^2 & 0 & 0 & 0 & 0\\
            0 & 0 & 0 & 0 & 0 & 0 & -6t & 0 & 0 & 0 & 0 & 0 & 3|t|^2 & 0 & 0 & 0\\
            0 & 0 & 0 & 0 & 0 & 6t & 0 & 0 & 0 & 0 & 0 & 0 & 0 & 3|t|^2 & 0 & 0\\
            0 & 0 & 0 & 0 & -6t & 0 & 0 & 0 & 0 & 0 & 0 & 0 & 0 & 0 & 3|t|^2 & 0\\
            0 & 0 & 0 & 0 & 0 & 0 & 0 & -6t & 0 & 0 & 0 & 0 & 0 & 0 & 0 & 3|t|^2\\
        \end{pmatrix}$}
        \caption{Matrix Representation of $\square_b^t$ on $\hks{3}$}
        \label{fig:matrix}
    \end{figure}
    With \textit{Mathematica's} Eigenvalue function, the eigenvalues were then calculated for these matrix representations. Our numerical results suggest that the smallest non-zero eigenvalue of $\boxbt$ on $\hks{2k-1}$ decreases as $k$ increases.  Conversely, the smallest non-zero eigenvalue of $\boxbt$ on $\hks{2k}$ increases with $k$. The smallest eigenvalue of $\hks{2k-1}$ is plotted for $1\leq k \leq 5$ and $0<|t|<1$ in Figure \ref{fig:plot}. It is apparent that $\lambda_{\min,1}\leq \lambda_{\min,3}\leq \lambda_{\min,5}\leq \lambda_{\min,7}\leq \lambda_{\min,9}$ where $\lambda_{\min,k}$ denotes the smallest non-zero eigenvalue of $\boxbt$ on $\hks{k}$. These initial numerical results suggest that $\lim_{k\to\infty} \lambda_{\min,2k-1}=0$ for $0<|t|<1$, which agrees with our final result.
    \begin{figure}[h]
        \begin{center}
        	\includegraphics[scale=.25]{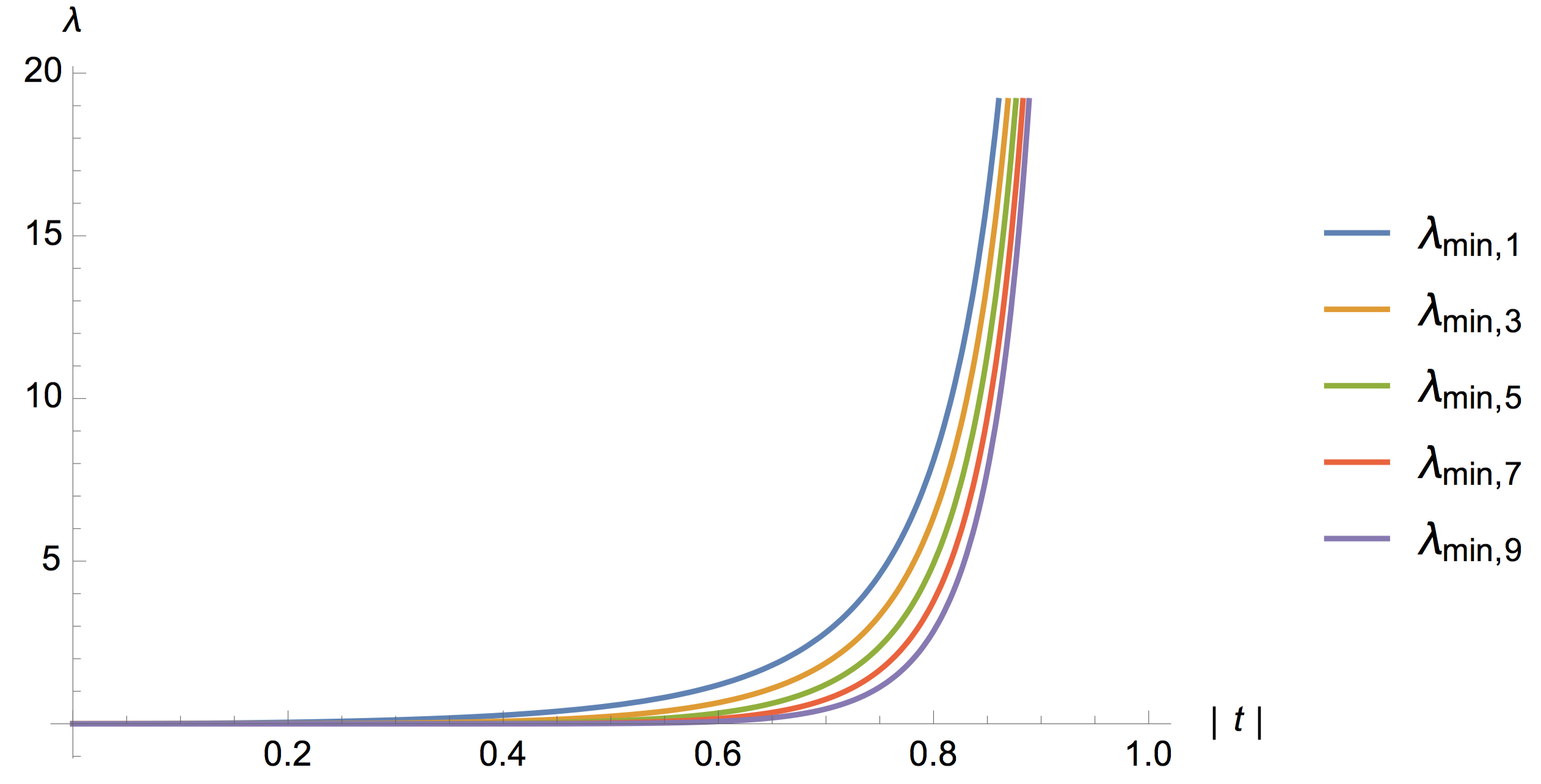}
        	\caption{Smallest Non-Zero Eigenvalues for $k=1,3,5,7,9$.}
	\label{fig:plot}
        \end{center}
    \end{figure}
    
%%%%%%%%%%%%%%%%%%%%%%%%%%%%%%%%%%%%%%%%%%%%%%%%%%    

    \section{Invariant Subspaces of $\hks{2k-1}$ under $\boxbt$}
    
In this section we fix $k\geq 1$ and work on $\hks{2k-1}$.    As we have seen, $\boxbt$ can be expanded in the following way:
	\begin{align}
		\boxbt &= -(\l + \conj{t} \lbar)\frac{1 + |t|^2}{(1- |t|^2)^2}(\lbar + t \l) \nonumber \\
		&= -h(\l \lbar + |t|^2 \lbar \l + t \l^2 + \conj{t} \lbar^2) \label{boxbtexpansion}
	\end{align}
	This is because of the linearity of $\l$ and $\lbar$. Now, we need the following property.
	\begin{lemma} 
	    \label{lbarsigmaorthogonal}
		If $\langle f_i, f_j \rangle=0$ for $i\not=j$ and $f_i,f_j\in \hpqs{0}{2k-1}$, then $\langle \lbar^{\sigma} f_i,  \lbar^{\sigma} f_j \rangle=0$ for $0\leq \sigma \leq 2k-1$.
	\end{lemma}
	\begin{proof} 
		Choose $f_i$ and $f_j$, $i\not=j$ from an orthogonal basis for $\hpqs{0}{2k-1}$. We show that for $0\leq \sigma\leq 2k-1$, $\lbar^{\sigma} f_1$ and $\lbar^{\sigma} f_j$ are orthogonal. To do this we use induction on $\sigma$. Suppose $\langle \lbar^{\sigma-1} f_i, \lbar^{\sigma-1} f_j \rangle =0$, and we show that $\langle \lbar^{\sigma} f_i, \lbar^{\sigma} f_j \rangle =0$. Note that,
		\begin{align*} 
			\langle \lbar^{\sigma} f_i, \lbar^{\sigma} f_j \rangle &=\langle \lbar^{\sigma-1} f_i, \l \lbar^{\sigma} f_j\rangle \\
			&=\langle \lbar^{\sigma-1} f_i, (\l \lbar)\lbar^{\sigma-1} f_j\rangle \\
			&=\langle \lbar^{\sigma-1} f_i, -\boxb \lbar^{\sigma-1} f_j\rangle.
		\end{align*}
		However since $\lbar^{\sigma-1} f_j\in \hpqs{\sigma-1}{2k-1-\sigma-1}$, we know that $\boxb \lbar^{\sigma-1} f_j = (\sigma)(2k-\sigma-2)\lbar^{\sigma-1}f_j$. Therefore, 
		\begin{align*} 
			\langle \lbar^{\sigma-1} f_i, -\boxb \lbar^{\sigma-1} f_j\rangle&=\langle \lbar^{\sigma-1} f_i,  -(\sigma)(2k-\sigma-2)\lbar^{\sigma-1}f_j\rangle \\
			&=-(\sigma)(2k-\sigma-2) \langle \lbar^{\sigma-1} f_i,  \lbar^{\sigma-1}f_j\rangle \\
			&= 0,
		\end{align*}
		by our induction hypothesis as desired. 
	\end{proof}
	With this, we first note that if $\{f_0, \dotsc, f_{2k-1}\}$ is an orthogonal basis for $\hpqs{0}{2k-1}$, then $\{\lbar^\sigma f_0, \dotsc, \lbar^\sigma f_{2k-1}\}$ is an orthogonal basis for $\hpqs{\sigma}{2k-1-\sigma}$. Now, we define the following subspaces of $\hks{2k-1}$.
	\begin{definition}
		\label{defoddvw}
		Suppose $\{f_0, \dotsc, f_{2k-1}\}$ is the orthogonal basis for $\hpqs{0}{2k-1}$. Then we define
		\begin{gather*}
			V_i = \Span\{f_i, \lbar^2 f_i,\ldots, \lbar^{2j-2} f_i, \dotsc, \lbar^{2k-2} f_i\}, \\
			W_i = \Span\{\lbar f_i, \lbar^3 f_i, \ldots, \lbar^{2j-1} f_i, \dotsc, \lbar^{2k-1} f_i\}.
		\end{gather*}
	\end{definition}
Denote the basis elements of $V_i$ by $v_{i,1}, \dotsc, v_{i,k}$ and for $W_i$ by $w_{i,1}, \dotsc, w_{i,k}$.  
We first note that since each bidegree space $\hpqs{p}{q} \subseteq \hks{2k-1}$ has $2k$ elements, we have $2k\; V_i$ spaces and $2k\; W_i$ spaces. We now note the following fact.

	\begin{theorem}
		$\displaystyle \bigoplus_{i=0}^{2k-1} V_i \oplus W_i = \hks{2k-1}$.
	\end{theorem}
	\begin{proof}
		We first note that by Theorem \ref{pkissumppq},
		\begin{align*}
		    \hks{2k-1} &= \bigoplus_{i=0}^{2k-1} \hpqs{i}{2k-1-i}
		\intertext{but by Lemma \ref{lbarsigmaorthogonal}, we see that this is really just}
	        &= \bigoplus_{i=0}^{2k-1} \lbar^i f_0 \oplus \dotsb \oplus \lbar^i f_{2k-1}. 
		\intertext{Manipulating this, we have}
		    &= \bigoplus_{i=0}^{2k-1} f_i \oplus \lbar f_i \dotsb \oplus \lbar^{2k-1} f_i \\
		    &= \bigoplus_{i=0}^{2k-1} f_i \oplus \lbar^2 f_i \oplus \dotsb \oplus \lbar^{2k-2} f_i \oplus \lbar f_i \oplus \lbar^3 f_i \oplus \dotsb \oplus \lbar^{2k-1} f_i \\
		    &= \bigoplus_{i=0}^{2k-1} V_i \oplus W_i, 
		\end{align*}
	    which is our goal.
	\end{proof}
	The advantage of constructing these spaces in the first place is due to the following fact.
	\begin{theorem}
		\label{vwinvariant}
		$\boxbt$ is invariant on $V_i$ and $W_i$.
	\end{theorem}
	\begin{proof}
		By equation \eqref{boxbtexpansion}, we have that
		\[
			\boxbt = -h(\l \lbar + |t|^2 \lbar \l + t \l^2 + \conj{t} \lbar^2)
		\]
		Since the fraction in front is a constant, we can ignore it and only consider the expression in the parentheses.
		Let $ f \in \hpqs{0}{2k-1} $, and define $v_\sigma = \lbar^\sigma f$ to be a basis element of either $V_i$ or $W_i$, since they have the same form. We first note that $ v_\sigma \in \hpqs{\sigma}{2k-1-\sigma} $. Then by our expansion we have that
		\[
			\boxbt v_\sigma = -h(\l \lbar v_\sigma + |t|^2 \lbar \l v_\sigma + t \l^2 v_\sigma + \conj{t} \lbar^2 v_\sigma)
		\]
		We already know $\l \lbar v_\sigma$ and $\lbar \l v_\sigma$ will simply be a multiple of $v_\sigma$, so we consider $\l^2 v_\sigma$ and $\lbar^2 v_\sigma$.
		\begin{subequations}
			\begin{align}
				\l^2 v_\sigma &= \l^2 \lbar^{\sigma} f \nonumber\\
				&= \l\left[\l \lbar \left[\lbar^{\sigma-1} f\right]\right] \nonumber\\
				&= -(\sigma)(2k - \sigma)\l \lbar \left[\lbar^{\sigma-2} f\right] \nonumber\\
				&= (\sigma)(\sigma - 1)(2k + 1 - \sigma)(2k - \sigma)\lbar^{\sigma-2} f \nonumber\\
				&= (\sigma)(\sigma - 1)(2k + 1 - \sigma)(2k - \sigma) v_{\sigma-2} \label{l2basiselem} \\
				~
				\lbar^2 v_\sigma &= \lbar^2 \left[\lbar^{\sigma} f \right] \nonumber\\
				&= \lbar^{\sigma+2} f \nonumber\\
				&= v_{\sigma+2} \label{lbar2basiselem}
			\end{align}
		\end{subequations}
		so we get multiples of $v_{\sigma-2}$ and $v_{\sigma+2}$. Relating this back to \(V_i\) and \(W_i\), we see that if $ \sigma = 2j-2 $, then $ \l^2 v_{i,j} $ is a multiple of $ v_{i,{j-1}} $, and $ \lbar^2 v_{i,j} $ is a multiple of $ v_{i,{j+1}} $. If $\sigma = 2j-1$, we get a similar result for $ w_{i,j} $. So we indeed have that $\boxbt$ is invariant on $V_i$ and $W_i$, and we are done.
	\end{proof}
	In light of this fact, we can consider $\boxbt$ not on the whole space $L^2(\3)$ or $\hks{2k-1}$, but rather on these $V_i$ and $W_i$ spaces. In fact, we actually have a representation of $\boxbt$ on these spaces.
	\begin{theorem}
		\label{oddtridiagonal}
		The matrix representation of $\boxbt$, $m(\boxbt)$, on $V_i$ and $W_i$ is tridiagonal, where $m(\boxbt)$ on $V_i$ is
		\[
			m(\boxbt) = h \begin{pmatrix}
				d_1 & u_1 &  &  & \phantom{\ddots} \\
				-\conj{t} & d_2 & u_2 &  & \phantom{\ddots} \\
				 & -\conj{t} & d_3 & \ddots & \\
				 &  & \ddots & \ddots & u_{k-1} \\
				\phantom{\ddots} &  &  & -\conj{t} & d_k
			\end{pmatrix}
		\]
		where $u_j = -t \cdot (2j)(2j-1)(2k-2j)(2k-1-2j)$ and $d_j = (2j-1)(2k+1-2j) + |t|^2 \cdot (2j-2)(2k+2-2j)$. For $W_i$, we get something similar:
		\[
			m(\boxbt) = h \begin{pmatrix}
			d_1 & u_1 &  &  & \phantom{\ddots} \\
			-\conj{t} & d_2 & u_2 &  & \phantom{\ddots} \\
			& -\conj{t} & d_3 & \ddots & \\
			&  & \ddots & \ddots & u_{k-1} \\
			\phantom{\ddots} &  &  & -\conj{t} & d_k
			\end{pmatrix}
		\]
		where $u_j = -t \cdot (2j+1)(2j)(2k-2j)(2k-1-2j)$ and $d_j = (2j)(2k-2j) + |t|^2 \cdot (2j-1)(2k+1-2j)$.
	\end{theorem}
	We note that the above definitions don't depend on $i$; in other words, each of these matrices are the same on $V_i$ and $W_i$, regardless of the choice of $i$.
	\begin{proof}
		Using equations \eqref{l2basiselem} and \eqref{lbar2basiselem}, along with Theorem \ref{boxbeigenvalue}, we can entirely describe the action of each piece of $\boxbt$ on a basis element $v_{i,j}$ or $w_{i,j}$:
		\begin{align*}
			-\l \lbar v_{i,j} &= (2j-1)(2k+1-2j)v_{i,j} & -\l \lbar w_{i,j} &= (2j)(2k-2j)w_{i,j} \\
			-\lbar \l v_{i,j} &= (2j-2)(2k+2-2j)v_{i,j}& -\lbar \l w_{i,j} &= (2j-1)(2k+1-2j)w_{i,j} \\
			\begin{split}
				-\l^2 v_{i,j} &= -(2j-2)(2j-3) \\
				&\hspace{2.1em}(2k+3-2j)(2k+2-2j)v_{i,{j-1}}
			\end{split} & 
			\begin{split}
			-\l^2 w_{i,j} &= -(2j-1)(2j-2) \\
			&\hspace{2.1em}(2k+2-2j)(2k+1-2j)w_{i,{j-1}}
			\end{split} \\
			-\lbar^2 v_{i,j}&= -v_{i,{j+1}} & -\lbar^2 w_{i,j} &= -w_{i,{j+1}}.
		\end{align*}
		By looking at it this way, we notice the tridiagonal structure. So with these observations, we can state that
		\begin{align*}
			\begin{split}
				\boxbt v_{i,j} = h \big(&-t \cdot (2j-2)(2j-3)(2k+3-2j)(2k+2-2j)v_{i,{j-1}} \\
				 &+ \left((2j-1)(2k+1-2j) + |t|^2 \cdot (2j-2)(2k+2-2j)\right)\!v_{i,j} \\ 
				 &-\conj{t} \cdot v_{i,{j+1}} \,\big)
			\end{split}& \\
			\begin{split}
				\boxbt w_{i,j} = h \big(&-t \cdot (2j-1)(2j-2)(2k+2-2j)(2k+1-2j) w_{i,{j-1}} \\
				&+ \left((2j)(2k-2j) + |t|^2 \cdot (2j-1)(2k-1-2j)\right)\!w_{i,j} \\ 
				&-\conj{t} \cdot w_{i,{j+1}} \,\big).
			\end{split}& 
		\end{align*}
		Now that we have this formula, we can find $m(\boxbt)$ on $V_i$ and $W_i$ by computing their effect on the basis vectors $v_{i,j}$ and $w_{i,j}$: when we do this for $V_i$, we get
		\begin{align*}
			d_j &= (2j-1)(2k+1-2j) + |t|^2 \cdot (2j-2)(2k+2-2j) \\
			u_{j-1}& = -t \cdot (2j-2)(2j-3)(2k+3-2j)(2k+2-2j)  \\
			&\implies u_j = -t \cdot (2j)(2j-1)(2k-2j)(2k-1-2j)  \\
		\end{align*}
		and for $W_i$, we get
		\begin{align*}
			d_j &= (2j)(2k-2j) + |t|^2 \cdot (2j-1)(2k-1-2j) \\
			u_{j-1} & = -t \cdot (2j-1)(2j-2)(2k+2-2j)(2k+1-2j) \\
			&\implies u_j = -t \cdot (2j+1)(2j)(2k-2j)(2k-1-2j).
		\end{align*}
		Finally, by factoring out $h$ and simply substituting each portion in we obtain the matrix representations above.
	\end{proof}
	An immediate consequence of this is that each $V_i$ subspace contributes the same set of eigenvalues to the spectrum of $\boxbt$, and similarly for each $W_i$. 
	Furthermore, we note that the matrices are rank $k$.
	Since the choice of $i$ does not change $m(\boxbt)$ on these spaces, we will fix an arbitrary $i$ and call the spaces $V$ and $W$ instead. 
	
%%%%%%%%%%%%%%%%%%%%%%%%%%%%%%%%%%%%%%%%%%%%%%%%%%

    \section{Bottom of the Spectrum of $\boxbt$} %Ravi
Now that we have a matrix representation for \(\boxbt\) on these \(V\) and \(W\) spaces inside $\hks{2k-1}$, we can begin to analyze their eigenvalues as $k$ varies. First, we go over some facts about tridiagonal matrices.
	\begin{theorem}
		\label{tridiagonalsymmetric}
		Suppose $A$ is a tridiagonal matrix,
		\[
			A = \begin{pmatrix}
			d_1 & u_1 &  &  & \phantom{\ddots} \\
			l_1 & d_2 & u_2 &  & \phantom{\ddots} \\
			& l_2 & d_3 & \ddots & \\
			&  & \ddots & \ddots & u_{k-1} \\
			\phantom{\ddots} &  &  & l_{k-1} & d_k
			\end{pmatrix}
		\]
		and the products \(u_i l_i > 0\) for $1\leq i\leq k$, then $A$ is similar to a symmetric tridiagonal matrix.
	\end{theorem}
	\begin{proof}
		One can verify that if
		\[
			S = \begin{pmatrix}
			1 &  &  &  & \phantom{\ddots} \\
			 & \sqrt{\frac{u_1}{l_1}} & & & \phantom{\ddots} \\
			 & & \sqrt{\frac{u_1 u_2}{l_1 l_2}} & & \phantom{\ddots} \\
			 & & & \ddots & \\
			\phantom{\ddots} &  &  &  & \sqrt{\frac{u_1 \dotsb u_{k-1}}{l_1 \dotsb l_{k-1}}}
			\end{pmatrix}
		\]
		then $A = SBS^{{-1}}$, where
		\[
			B = \begin{pmatrix}
			d_1 & \sqrt{u_1 l_1} &  &  & \phantom{\ddots} \\
			\sqrt{u_1 l_1} & d_2 & \sqrt{u_2 l_2} &  & \phantom{\ddots} \\
			& \sqrt{u_2 l_2} & d_3 & \ddots & \\
			&  & \ddots & \ddots & \sqrt{u_{k-1} l_{k-1}} \\
			\phantom{\ddots} &  &  & \sqrt{u_{k-1} l_{k-1}} & d_k
			\end{pmatrix}
		\]
		Therefore, $A$ is similar to a symmetric tridiagonal matrix.
	\end{proof}
	Another special property of tridiagonal matrices is the continuant.
	\begin{definition}
		\label{continuant}
		Let $A$ be a tridiagonal matrix, like the above. Then we define the \textit{continuant} of $A$ to be a recursive sequence: $f_1 = d_1$, and $f_{i} = d_{i-1} f_{i-1} - u_{i-2} l_{i-2} f_{i-2}$, where $f_0 = 1$.
	\end{definition}
	The reason we define this is because $\det(A) = f_k$. In addition, if we denote $A_i$ to mean the square submatrix of $A$ formed by the first $i$ rows and columns, then $\det(A_i) = f_i$. \\
	With this background, we will now start analyzing $\boxbt$ on $W$. \\
	
	To get bounds on the eigenvalues, we will invoke the Cauchy interlacing theorem, see \cite{Cauchy}.
	\begin{theorem}
		\label{cauchyinterlacing}
		Suppose $A$ is an $n\times n$ Hermitian matrix of rank $n$, and $B$ is an $n-1 \times n-1$ matrix minor of $A$. If the eigenvalues of $A$ are $\lambda_1 \le \dotsb \le \lambda_n$ and the eigenvalues of $B$ are $\nu_1 \le \dotsb \le \nu_{n-1}$, then the eigenvalues of $A$ and $B$ interlace:
		\[
			0< \lambda_1 \le \nu_1 \le \lambda_2 \le \nu_2 \le \dotsb \le \lambda_{n-1} \le \nu_{n-1} \le \lambda_n
		\]
	\end{theorem}
	Now, we can get an intermediate bound on the smallest eigenvalue.
	\begin{theorem}
		\label{boundweigenvals}
		Suppose $A$ is a Hermitian matrix of rank $n$, and  $\lambda_1 \le \dotsb \le \lambda_n$ are its eigenvalues. Then
		\[
			\lambda_1 \le \frac{\det(A)}{\det(A_{k-1})}
		\]
		where $A_{k-1}$ is $A$ without the last row and column.
	\end{theorem}
	\begin{proof}
		Since $A_{k-1}$ is a $k-1 \times k-1$ matrix minor of $A$, we can apply the Cauchy interlacing theorem. If the eigenvalues of $A_{k-1}$ are $\nu_1 \le \dotsb \le \nu_{k-1}$, then
		\[
			\lambda_1 \le \nu_1 \le \lambda_2 \le \nu_2 \le \dotsb \le	\lambda_{n-1} \le \nu_{n-1} \le \lambda_n
		\]
		Now, we claim that
		\[
			\lambda_1 \det(A_{k-1}) \le \det(A)
		\]
		To see why this is true, first observe that the determinant of a matrix is simply the product of all its eigenvalues. In particular,
		\[
			\lambda_1 \det(A_{k-1}) = \lambda_1 \nu_1 \dotsb \nu_{k-1}
		\]
		But we can simply apply the Cauchy interlacing theorem: since $\nu_1 \le \lambda_2$, $\nu_2 \le \lambda_3$, and so on, we get
		\begin{align*}
			\lambda_1 \nu_1 \dotsb \nu_{k-1} &\le \lambda_1 \lambda_2 \dotsb \lambda_k \\
			&= \det(A)
		\end{align*}
		so the claim is proven. Now, dividing both sides by $\det A_{k-1}$,
		\[
			\lambda_1 \le \frac{\det(A)}{\det(A_{k-1})}
		\]
		as desired.
	\end{proof}

	Since $m(\boxbt)$ on $W$ satisfies the conditions of Theorem \ref{tridiagonalsymmetric}, we find it is similar to this Hermitian tridiagonal matrix:
	\[
	    A = \begin{pmatrix}
				a_1 + b_1|t|^2 & c_1|t| &  &  & \phantom{\ddots} \\
				c_1|t| & a_2 + b_2|t|^2 & c_2|t| &  & \phantom{\ddots} \\
				 & c_2|t| & a_3 + b_3|t|^2 & \ddots & \\
				 &  & \ddots & \ddots & c_{k-1}|t| \\
				\phantom{\ddots} &  &  & c_{k-1}|t| & a_k + b_k|t|^2 
		\end{pmatrix}
	\]
	where $a_i = (2i)(2k-2i)$, $b_i = (2i-1)(2k+1-2i)$, and $c_i = \sqrt{(2i+1)(2i)(2k-2i)(2k-1-2i)}$. Note that we are ignoring the constant $h$ for now, which we will add back later. If we can find $\det(A_i)$, then by Theorem \ref{boundweigenvals} we can get a closed form for the bound on the smallest eigenvalue. With the following lemma, this is possible:
	\begin{lemma}
	    \label{crossdiagonal}
	    $a_i b_{i+1} = c_i^2$
	\end{lemma}
	\begin{proof}
	   We can simply work through the formulas to figure this out: $a_i = (2i)(2k-2i)$, $b_{i+1} = (2i+1)(2k-1-2i)$, and $c_i^2 = (2i+1)(2i)(2k-2i)(2k-1-2i)$. The products clearly match up.
	\end{proof}
	
    \begin{theorem}
        \label{detoddwmatrix}
        The determinant of $A_i$ is
        \[
            \begin{split}
                \det(A_i) &= a_1 a_2 \dotsb a_{i-1} a_i \\
                &+ b_1 a_2 \dotsb a_{i-1} a_i |t|^2 \\
                &+ \dotsb \\
                &+ b_1 b_2 \dotsb b_{i-1} a_i |t|^{2i-2} \\
                &+ b_1 b_2 \dotsb b_{i-1} b_i |t|^{2i}
            \end{split}
        \]
    \end{theorem}
    In each row, we replace a particular $a_j$ with $b_j$, and multiply by $|t|^2$. Note that if $i = k$, then $a_k = 0$ and all terms but the last term are 0.
    \begin{proof}
        We will prove this using strong induction on $i$. The base case is $i = 1$, where $\det(A_1) = a_1 + b_1 |t|^2$, which does indeed match up with our formula. Now, assume the formula works for $A_{i-1}$ and $A_i$. We need to show that the formula works for $A_{i+1}$. Using the formula for the continuant, we get
        \begin{align*}
            \det(A_{i+1}) &= (a_{i+1} + b_{i+1} |t|^2) \det(A_i) - c_i^2 |t|^2 \det(A_{i-1}) \\
            \intertext{Now, use Lemma \ref{crossdiagonal}:}
            &= (a_{i+1} + b_{i+1} |t|^2) \det(A_i) - a_i b_{i+1} |t|^2 \det(A_{i-1}) \\
        \end{align*}
        Now, we use our induction hypothesis: 
        \begin{alignat*}{2}
            &=&&\;(a_{i+1} + b_{i+1} |t|^2) (a_1 a_2 \dotsb a_i + b_1 a_2 \dotsb a_i |t|^2 + \dotsb + b_1 b_2 \dotsb b_i |t|^{2i}) \\
            &&&- a_i b_{i+1} |t|^2 (a_1 a_2 \dotsb a_{i-1} + b_1 a_2 \dotsb a_{i-1} |t|^2 + \dotsb + b_1 b_2 \dotsb b_{i-1} |t|^{2i-2}) \\
            &=&&\; a_1 a_2 \dotsb a_{i+1} + b_1 a_2 \dotsb a_{i+1} |t|^2 + \dotsb + b_1 b_2 \dotsb b_i a_{i+1} |t|^{2i} \\
            &&&+  a_1 a_2 \dotsb a_i b_{i+1}|t|^2 + b_1 a_2 \dotsb a_i b_{i+1}|t|^4 + \dotsb + b_1 b_2 \dotsb b_{i-1} a_i b_{i+1} |t|^{2i+2} + b_1 b_2 \dotsb b_{i+1} |t|^{2i+2} \\
            &&&- a_1 a_2 \dotsb a_i b_{i+1}|t|^2 - b_1 a_2 \dotsb a_i b_{i+1}|t|^4 - \dotsb - b_1 b_2 \dotsb b_{i-1} a_i b_{i+1} |t|^{2i+2} \\
            &=&&\;a_1 a_2 \dotsb a_{i+1} + b_1 a_2 \dotsb a_{i+1} |t|^2 + \dotsb + b_1 b_2 \dotsb b_i a_{i+1} |t|^{2i} + b_1 b_2 \dotsb b_{i+1} |t|^{2i+2}
        \end{alignat*}
        which is the formula for $A_{i+1}$, and we are done.
    \end{proof}
   
    With this knowledge, we are finally able to prove our theorem.
    \begin{theorem}
        $0 \in \essspec(\boxbt)$
    \end{theorem}
    
    \begin{proof}
        By Theorem \ref{detoddwmatrix}, we have that on $W$ in $\hks{2k-1}$, $m(\boxbt)$ is similar to
        \[
    	    A = h \begin{pmatrix}
        		a_1 + b_1|t|^2 & c_1|t| &  &  & \phantom{\ddots} \\
    			c_1|t| & a_2 + b_2|t|^2 & c_2|t| &  & \phantom{\ddots} \\
    			& c_2|t| & a_3 + b_3|t|^2 & \ddots & \\
    		    &  & \ddots & \ddots & c_{k-1}|t| \\
    			\phantom{\ddots} &  &  & c_{k-1}|t| & a_k + b_k|t|^2 
    		\end{pmatrix}
        \]
        where $a_j = (2j)(2k-2j)$, $b_j = (2j-1)(2k+1-2j)$, and $c_j = \sqrt{(2j+1)(2j)(2k-2j)(2k-1-2j)}$. Now, by Theorem \ref{boundweigenvals} above, we know that
        \[
        	\lambda_{\min} \leq \frac{\det(A)}{\det(A_{k-1})}.
        \]
        Recall that $A_{k-1}$ denotes the submatrix formed by deleting the last row and column of the $k\times k$ matrix $A$. To show that $0 \in \essspec(\boxbt)$, we want to show that $\det(A)/\det(A_{k-1}) \to 0$ as $k \to \infty$. For this purpose we find an upper bound for $\det(A)/\det(A_{k-1})$ and show that this converges to $0$. Notice that Theorem \ref{detoddwmatrix} implies that,
        \begin{align}\label{eqn:ratioDets}
        \frac{\det(A)}{\det(A_{k-1})} &=h\frac{b_1 b_2  \dots b_{k-1} b_k |t|^{2k}}{a_1 a_2 \dots a_{k-1}+b_1 a_2 \dots a_{k-1}|t|^2+ b_1 b_2 \dots a_{k-1}|t|^4+\dots b_1 b_2 \dots b_{k-1}|t|^{2k-2} } \nonumber \\
         &\leq h\frac{b_1 b_2  \dots b_{k-1} b_k |t|^{2k}}{a_1 a_2 \dots a_{k-1}}. 
         \end{align}
        since, $a_j,b_j,|t|>0$. Now using the formulas for $a_j$ and $b_j$, notice that $($\ref{eqn:ratioDets}$)$ can be written as
        \[
            h(2k-1)|t|^{2k} \prod_{j=1}^{k-1} \frac{(2j+1)(2k-2j-1)}{(2j)(2k-2j)}.
        \]
        However, we know that for all $k$ and $1\leq j \leq k-1,$
        \[ 
            \frac{(2k-2j-1)}{(2k-2j)} <1, 
        \]
        and so,
        \[ 
            h(2k-1)|t|^{2k} \prod_{j=1}^{k-1} \frac{(2j+1)(2k-2j-1)}{(2j)(2k-2j)} \leq h(2k-1)|t|^{2k} \prod_{j=1}^{k-1} \frac{(2j+1)}{(2j)}=h(2k-1)|t|^{2k} \prod_{j=1}^{k-1} 1 + \frac{1}{2j}.
        \]
        Furthermore, we have
        \[
            h(2k-1)|t|^{2k} \prod_{j=1}^{k-1} 1 + \frac{1}{2j} \leq h(2k-1)|t|^{2k} \exp \left( \sum_{j=1}^{k-1} \frac{1}{2j} \right).
        \]
        Note that 
        \[
	        \sum_{j=1}^{k-1} \frac{1}{2j} \le \frac{1}{2} \ln k
        \]
        so our expression becomes
        \[
	    \frac{\det(A)}{\det(A_{k-1})}     \le h(2k-1)|t|^{2k} \exp \left( \frac{1}{2} \ln k \right) = h(2k-1)\sqrt{k}\,|t|^{2k}
        \]
        and our problem reduces to showing that $\lim\limits_{k\to \infty} h(2k-1)\sqrt{k}|t|^{2k} = 0$. We note that $h$ is a constant and $|t|<1$; therefore, by L'Hospital's rule the last expression indeed goes to $0$.

        Finally, we have,
        \[
        0\leq \lim_{k\to\infty} \lambda_{\min} \leq \lim_{k\to\infty} \frac{\det(A)}{\det(A_{k-1})} \leq \lim_{k\to\infty} h(2k-1)\sqrt{k}\,|t|^{2k} =0,
        \]
        and so $\lambda_{\min}\to0$. Hence $0\in \essspec(\boxbt).$
    \end{proof}

We note that by the discussion in the introduction, this means that the CR-manifold $(\l_t,\3)$ is not embeddable into any $\C^N$.

    \section*{Acknowledgements}
	This research was conducted at the NSF REU Site (DMS-1659203) in Mathematical Analysis and Applications at the University of Michigan-Dearborn. We would like to thank the National Science Foundation, the College of Arts, Sciences, and Letters, the Department of Mathematics and Statistics at the University of Michigan-Dearborn, and Al Turfe for their support. We would also like to thank John Clifford, Hyejin Kim, and the other participants of the REU program for fruitful conversations on this topic.
%\bibliographystyle{alpha}
%\bibliography{paper}

\end{document}